\documentclass[10pt,a4paper]{article}
\usepackage{amsmath,amssymb,amsthm,amsfonts,graphicx,color,bm,indentfirst}
\usepackage[hypertex]{hyperref}
\usepackage{geometry}
\def\theequation{\thesection.\@arabic\c@equation}
\makeatother

\renewcommand{\theequation}{\thesection.\arabic{equation}}
\newtheorem{lemma}{Lemma}[section]

\newtheorem{proposition}{Proposition}[section]
\newtheorem{corollary}{Corollary}[section]
\newtheorem{remark}{Remark}[section]
\newtheorem{theorem}{Theorem}[section]
\newtheorem{conjecture}{Conjecture}[section]

\newcommand{\F}{\frac}
\def\R{\mathbb{R}}

\allowdisplaybreaks

\title
{Infinitely many solutions to a fractional nonlinear Schr\"{o}dinger equation}

\author{  Liping Wang \footnotemark[2] \qquad and \qquad Chunyi Zhao\footnotemark[1]  \medskip \\
{\small Department of Mathematics,} \\ {\small Shanghai Key Laboratory of Pure Mathematics and Mathematical Practice, }\\{\small
East China Normal University, Shanghai, 200241, China} }

\begin{document}

\maketitle

\footnotetext[2]{lpwang@math.ecnu.edu.cn}
\footnotetext[1]{Corresponding author. cyzhao@math.ecnu.edu.cn }

\begin{abstract}
This paper considers the fractional Schr\"{o}dinger equation
\begin{equation}\label{abstract}
(-\Delta)^s u + V(|x|)u-u^p=0, \quad u>0, \quad u\in H^{2s}(\R^N)
\end{equation}
where $0<s<1$, $1<p<\frac{N+2s}{N-2s}$, $V(|x|)$ is a positive potential and $N\geq 2$.
We show that if $V(|x|)$ has the following expansion:
\[
V(|x|)=V_0 + \frac{a}{|x|^m} + o\left(\frac{1}{|x|^m}\right) \qquad \mbox{as} \ |x| \rightarrow +\infty,
\]
in which the constants are properly assumed,
then (\ref{abstract}) admits infinitely many non-radial solutions, whose energy can be made arbitrarily large. This is the first result for fractional Schr\"{o}dinger equation.
The $s=1$ case corresponds to the known result in Wei-Yan \cite{WY}.
\end{abstract}


\section{Introduction and main results}
\setcounter{equation}{0}

The fractional Schr\"odinger equation is a fundamental equation of fractional quantum mechanics.
The nonlinear fractional nonlinear Schr\"{o}dinger equation is as follows:
\begin{equation}\label{sd}
i\psi_t=(-\Delta)^s\psi + \widetilde{V}(x) \psi -|\psi|^{p-1}\psi
\end{equation}
where $(-\Delta)^s$ ($0<s<1$) denotes the classical fractional Laplacian, $\tilde{V}$ is a bounded potential and $p>1$. \par

We are interested in finding {\it standing wave solutions}, which are
solutions of the form $\psi(x,t)=u(x)e^{i\lambda t}$ with the function $u$ real-valued. Let $V(x)=\widetilde{V}(x) + \lambda$, then $\psi$ is a solution of (\ref{sd}) if and only if $u$ solves the following
equation
\begin{equation}\label{fsd}
(-\Delta)^su + V(x)u -|u|^{p-1}u=0 \qquad \mbox{in }\R^N.
\end{equation}
A similar  problem to (\ref{fsd}) is the following fractional scalar field equation
\begin{equation}\label{sfe}
(-\Delta)^s u + u=Q(x)u^p, \qquad u>0 \qquad \mbox{in} \ \R^N.
\end{equation}
It is also absorbing to study the singularly perturbed problem
\begin{equation}\label{ssd}
\varepsilon^{2s}(-\Delta)^su + V(x)u -|u|^{p-1}u=0 \qquad \mbox{in} \ \R^N
\end{equation}
or
\begin{equation}\label{ssfe}
\varepsilon^{2s}(-\Delta)^s u + u=Q(x)u^p, \qquad u>0 \qquad \mbox{in}  \ \R^N
\end{equation}
where $\varepsilon>0$ is a small parameter. The natural place to look for solutions that decay at infinity is the space $H^{2s}(\R^N)$
of all functions $u\in L^2(\R^N)$ such that
\begin{equation*}
  \int_{\R^N} (1+|\xi|^{4s})|\hat u(\xi)|^2 \mathrm d\xi < \infty,
\end{equation*}
where $\widehat{}$ denotes the Fourier transform. The fractional Laplacian $(-\Delta)^s u$ for $u\in H^{2s}(\R^N)$ is defined by
\begin{equation*}
  \widehat{(-\Delta)^s u} (\xi) = |\xi|^{2s} \hat {u}(\xi) .
\end{equation*}\par

For (\ref{fsd})--(\ref{ssfe}), an interesting problem is to find solutions with a spike pattern concentrating around some points.
As for the standard case $s=1$ of (\ref{ssd}) or (\ref{ssfe}), this has been the topic of
many works relating the concentration points with critical points of the potential, starting in 1986 from the pioneering work Floer-Weinstein \cite{FW}.
Later many works show that the number of the critical points of $V(x)$ (or $Q(x)$) (see for example \cite{ABC, CNY1, CNY2, DF1, DF2, DF3, DF4, Oh,
Wang}), the type of the critical points of $V(x)$ (or $Q(x)$) (see for example \cite{DY, KW, NY,Z}), and the topology of the level set  $V(x)$ (or $Q(x)$) (see for example \cite{AMN1, AMN2, DLY, FM}), can
effect the number of solutions of (\ref{ssd}) (or (\ref{ssfe})).
It is now known that when the parameter $\varepsilon $ goes to zero, the number of the solutions may tend to infinity.
For the $s=1$ case of (\ref{fsd}) or (\ref{sfe}), in 2010 Wei-Yan \cite{WY} get a multiplicity result
under some symmetry assumption of $V(x)$ near the  infinity.
Recently we are told that del Pino-Wei-Yao \cite{DWY} get the similar result with a weaker symmetry assumption on $V(x)$.\par

As to the fractional case $0< s<1$, very few is known.  Recently D\'avila-del Pino-Wei \cite{DDW}
obtained the first result of spike pattern for the fractional Schr\"{o}dinger equation (\ref{ssd}) with $1 <p < \frac{N+2s}{N-2s}$.
{\it A natural question is can we get multiplicity result for (\ref{fsd}) (or(\ref{sfe})) with $0< s < 1$?
What is the situation in the fractional case?} In this paper we will give an affirmative answer! \par

This paper is concerned about the following fractional Laplacian problem
\begin{equation}\label{p}
(-\Delta)^su + V(|x|)u -|u|^{p-1}u=0, \quad u>0, \qquad u\in H^{2s}(\R^N)
\end{equation}
where $0<s<1$, $1<p<\frac{N+2s}{N-2s}$ and $N\geq 2$. We suppose that $V(x)$ satisfies the following assumption. \\
\textbf{Assumption $\bm{\mathcal V}$.}
$V$ is positive and radially symmetric, i.e. $V(x)=V(|x|)>0$ and there are constants $a>0$ and $V_0>0$ such that
\begin{equation}
V(|x|)=V_0 + \frac{a}{|x|^m} + o\left( \frac1{|x|^m}\right), \qquad \mbox{as} \ |x| \rightarrow +\infty,
\end{equation}
where
\begin{equation}\label{1.8}
\max\left\{0,\ (N+2s)\left[1-(p-1)N-2ps + \max\left\{s, \ p-\frac{N}2\right\}\right]\right\} < m < N+2s.
\end{equation}
Without loss of generality, we may assume $V_0=1$ for the sake of simplicity.\par
It's easy to see that
\[
\left[1-(p-1)N-2ps + \max\left\{s, \ p-\frac{N}2\right\}\right]<1 \qquad
\mbox{for  any } p>1, \ s \in (0, 1).
\]
By direct computations we find that in three dimension case, if
\[
1 + \frac{1-s}{3+2s}<p<\frac{3+2s}{3-2s}, \qquad \frac16 < s \le
\frac12,
\]
then we just need that $m\in (0, N+2s)$.

The aim of this paper is to obtain {\it infinitely many non-radial positive solutions} to (\ref{p}), {\it whose energy may be arbitrarily large}.
Our main result in this paper is stated in the following theorem.

\begin{theorem}\label{t1}
If $V(|x|)$ satisfies the assumption $\mathcal V$, then the problem (\ref{p}) admits infinitely many non-radial positive solutions. Moreover, the energy of these solutions may be arbitrarily large.
\end{theorem}

\begin{remark}
The condition on potential $V(x)$ is  more general  than that
of $V$ in \cite{WY} for $s=1$.  The main reason is that in our case
we can deduce the exact relationship between the radius and the
number of spikes, while in \cite{WY}, the authors can't solve it
exactly using the leading terms of energy.
\end{remark}

We believe that the symmetry on $V$ is technical and then make the following conjecture.

\begin{conjecture}
Problem (\ref{p}) has infinitely many solutions if there are constants $a>0, m \in(0, N+2s)$ and $V_0>0$, such that
\[
V(x)=V_0 + \frac{a}{|x|^m} + o\left( \frac1{|x|^m}\right), \qquad \mbox{as} \ |x| \rightarrow +\infty.
\]
\end{conjecture}

\begin{remark}
Using the same argument, we can prove that if
\[
Q(|x|)=Q_0 - \frac{a}{|x|^m}+ o\left( \frac1{|x|^m}\right) \qquad \mbox{as} \ |x| \rightarrow +\infty,
\]
where the constants are similarly assumed, then  problem (\ref{sfe}) has infinitely many positive non-radial solutions.
\end{remark}

Before close this introduction, let us outline the main idea in the proof of Theorem \ref{t1}.
Our aim is to construct solutions with a large number of bumps near the infinity. Since
\[
\lim_{|x| \rightarrow +\infty} V(|x|)=1,
\]
we will use the solution of
\begin{equation}\label{p1}
(-\Delta)^su +u -|u|^{p-1}u=0, \qquad u>0, \qquad u\in H^{2s}(\R^N)
\end{equation}
to build up the approximate solution for problem (\ref{p}).  It is known (see for instance \cite{FLS}) the existence of a positive, radial least energy solution $w(x)$, which gives the lowest possible value for the energy
\[
J_1(v)=\frac12\int_{\R^N} v(-\Delta)^s v + \frac12\int_{\R^N} v^2 - \frac1{p+1}\int_{\R^N} |v|^{p+1}
\]
among all nontrivial solutions of (\ref{p1}).  An important property, which has  been proven recently by Frank-Lenzman-Silvestre \cite{FLS} (see also \cite{AT, FL}), is that there exists a radial
least energy solution which is non-degenerate, in the sense that the space of solutions of the equation
\begin{equation}
(-\Delta)^s \phi + \phi -pw^{p-1}\phi=0, \qquad \phi \in H^{2s}(\R^N)
\end{equation}
consists exactly of the linear combinations of the translation-generators $\frac{\partial w}{\partial x_j}, j=1, \ldots, N.$ Also we have the following  behavior for $w(x)$ (\cite{FLS}):
\begin{equation}\label{expw}
 w'(|x|)<0; \qquad w(|x|)=\frac{A}{|x|^{N+2s}}\left(1+o(1)\right), \qquad A>0, \qquad \mbox{ as} \quad  |x| \rightarrow +\infty.
\end{equation}

Let
\[
 q_j = \left(r\cos\frac{2(j-1)\pi}{k}, r\sin\frac{2(j-1)\pi}{k}, \bm{0}\right), \qquad j=1, \ldots, k,
\]
where $\bm 0$ is the zero vector in $\R^{N-2}, r \in \left[\frac1{C_0}k^{\frac{N+2s}{N+2s-m}}, C_0k^{\frac{N+2s}{N+2s-m}} \right]$ for large positive constant $C_0$ independent of $k$.

 Set $x=(x', x'')$, $x \in \R^2$, $x'' \in \R^{N-2}$. Define
\begin{multline*}
H_s=\Big\{u ~\Big|~ u \in H^{2s}(\R^N),\ u \ \mbox{is even in}\ x_h\ (h=2, \ldots, N)\ \text{and}\\
 u\left(r\cos\theta, r\sin\theta, x''\right)=u\left(r\cos\left(\theta + \frac{2\pi j}{k}\right), r\sin\left(\theta + \frac{2\pi j}{k}\right), x''\right),\ j=1, \ldots, k-1   \Big\}.
\end{multline*}
Define
\[
W(x)=\sum\limits_{j=1}^k w(x-q_j),
\]
then Theorem \ref{t1} is a direct consequence of the following result.

\begin{theorem}\label{t2}
Suppose $V(|x|)$ satisfies the assumption $\mathcal V$. Then there is an integer $k_0>0$, such that for any integer $k \ge k_0$,
Problem (\ref{p}) has a solution $u_k$ of the form
\[
u_k(x) =W(x) + \varphi(x),
\]
where $\varphi(x) \in H_s$ and the energy at $u_k$  goes to infinity as $k $ goes to infinity.
\end{theorem}

\begin{remark}
Note that there is no parameter in the problem (\ref{p}). Using the number of spikes as parameter, we  get the {\bf{ first multiplicity result}} for fractional nonlinear Schr\"{o}dinger equation,
{\bf which seems a  new phenomenon} for fractional nonlinear Schr\"{o}dinger equation.
\end{remark}

\begin{remark}
Since  the approximate solution has polynomial decay, we should deal with every term carefully in the calculous which makes our proof a little bit complicated. By the way, in \cite{WY}, the approximation has exponential decay.
\end{remark}

The paper is organized as follows. In Section \ref{s2}, we introduce some preliminaries. In Section \ref{s3}, the ansatz is established. In Section \ref{s4}, we deal with the corresponding
linearized problem. In Section \ref{s5}, the nonlinear problem is considered and the proof of Theorem \ref{t2} is given. Finally some important estimates and the expansion of the energy are stated in Section
\ref{s6}.\par

\textbf{Notations.} In what follows, the symbol $C$ always denotes a various constant independent of $k$.

\section{Preliminaries}\label{s2}
\setcounter{equation}{0}

In this section, we get a useful a-priori estimate for a related linear equation.

Let $0 < s <1$. Various definitions of the fractional Laplacian $(-\Delta)^s \varphi$ of a function $\varphi$ defined in $\R^N$ are available, depending on its regularity and growth properties, see
for example \cite{DDW}. A useful (local) representation given by Caffarelli and Silvestre \cite{CS}, is via the following boundary value problem in the half space $\R_+^{N+1}=\{(x, y)~|~ x\in \R^N, y>0 \}$:
\[
\nabla \cdot \left(y^{1-2s} \nabla \tilde{\varphi}\right)=0 \quad \mbox{in} \ \R_+^{N+1}, \qquad \tilde{\varphi}(x, 0)=\varphi(x) \quad \mbox{on} \ \R^N.
\]
Here $\tilde{\varphi}$ is the $s$-harmonic extension of $\varphi$, explicitly given as a convolution integral with the $s$-Poisson kernel $p_s(x, y)$,
\[
\tilde{\varphi}(x, y)=\int_{\R^N} p_s(x-z, y)\varphi(z) dz,
\]
where
\[
p_s(x, y)=c_{N,s}\frac{y^{4s-1}}{(|x|^2 + |y|^2)^{\frac{N-1+4s}{2}}}
\]
and $c_{N,s}$ achieves $\int_{\R^N} p_s(x, y)dx=1$. Then under suitable regularity, $(-\Delta)^s \varphi$ is the Dirichlet-to-Neumann map for this problem, that is
\begin{equation}\label{DN}
(-\Delta)^s \varphi(x) =\lim_{y \rightarrow 0^+} y^{1-2s}\partial_y \tilde{\varphi}(x,y).
\end{equation}

For $m>0$ and $g \in L^2(\R^N)$, let us consider now  the equation
\[
(-\Delta)^s \varphi+m \varphi =g \qquad \mbox{in} \ \R^N.
\]
Then in terms of Fourier transform, for $\varphi \in L^2(\R^N)$, this problem reads
\[
\left(|\xi|^{2s} + m   \right)\hat{\varphi}=\hat{g}
\]
and has a unique solution $\varphi \in H^{2s}(\R^N)$ given by the convolution
\begin{equation}\label{conv}
\varphi(x)=T_m(g):=\int_{\R^N} k(x-z)g(z)dz
\end{equation}
where the Fourier transform of $k$ is
\[
\hat{k}(\xi) =\frac{1}{|\xi|^{2s}+m}.
\]
Then we have the following  main properties of the fundamental solution $k(x)$ (see for example \cite{FLS, FQT}): $k(x)$ is radially symmetric and positive, $k \in C^{\infty}(\R^N
\setminus \{0\})$ satisfying\\
\[
\begin{split}
&(i) \qquad |k(x)| + |x||\nabla k(x)| \le \frac{C}{|x|^{N-2s}} \qquad \quad \mbox{for all} \quad |x| \le 1;\\
&(ii) \qquad \lim_{|x|\rightarrow \infty} k(x)|x|^{N+2s}=\alpha >0;\\
&(iii) \qquad  |x||\nabla k(x)| \le \frac{C}{|x|^{N+2s}} \qquad \qquad \qquad \mbox{for all} \quad |x| \ge 1.
\end{split}
\]

Using (\ref{DN}) written in weak form, $\varphi$ can be characterized by $\varphi(x) =\tilde{\varphi}(x, 0)$ in trace sense, where $\tilde{\varphi} \in H$ is the unique solution of
\begin{equation}\label{ws}
\iint_{\R_+^{N+1}} \nabla \tilde{\varphi}(x, y) \cdot \nabla \phi(x, y) y^{1-2s}\mathrm dx\mathrm dy + m\int_{\R^N} \varphi(x) \phi(x,0) dx=\int_{\R^N} g(x)\phi(x,0)dx,
\end{equation}
for all $ \phi \in H$, where $H$ is the Hilbert space of functions $\phi \in H^1_\text{loc}(\R_+^{N+1})$ such that
\[
\|\phi\|_H^2:=\iint_{\R_+^{N+1}} | \nabla \phi(x, y)|^2 y^{1-2s}dxdy + m\int_{\R^N}  |\phi(x,0)|^2 dx < +\infty,
\]
or equivalent the closure of the set of all functions in $C_c^\infty(\overline{\R_+^{N+1}})$ under this norm.

For our purpose, we need the following four lemmas, see \cite{DDW}:
\begin{lemma}\label{2.1}
Let $g \in L^2(\R^N)$. Then the unique solution $\tilde{\varphi} \in H$ of the problem (\ref{ws}) is given by the s-harmonic extension of the function $\varphi=T_m(g)$.
\end{lemma}

\begin{lemma}\label{2.2}
Let $0 \le \mu < N+2s$. Then there exists a  positive constant $C$ such that
\[
\|(1 + |x|)^\mu T_m(g)\|_{L^\infty(\R^N)} \le C \|(1 + |x|)^\mu g\|_{L^\infty(\R^N)}.
\]
\end{lemma}

\begin{lemma}\label{2.3}
Assume that $g \in L^2(\R^N)\cap L^\infty(\R^N)$. Then the following holds: if $\varphi=T_m(g)$ then there is a $C>0$ such that
\begin{equation}
\sup\limits_{x \neq y} \frac{\left|\varphi(x) - \varphi(y)\right|}{|x-y|^\beta} \le C \|g\|_{L^\infty(\R^N)}
\end{equation}
where $\beta =\min\{ 1, 2s \}$.
\end{lemma}

\begin{lemma}\label{2.4}
 Let $\varphi \in H^{2s}$ be the solution  of
\begin{equation}\label{mp}
(-\Delta)^s \varphi + W(x) \varphi =g \qquad \mbox{in}\  \R^N
\end{equation}
 with bounded potential $W$. If $\inf_{x \in \R^N} W(x)=: m>0$,  $g \ge 0$. Then $\varphi \ge 0$ in
 $\R^N$.
\end{lemma}

Using these lemmas, we obtain an a-priori estimate for any solution $\varphi \in L^2(\R^N)\cap L^\infty(\R^N)$ of (\ref{mp}).
\begin{lemma}\label{apriori}
Let $W$ be a continuous function, and assume that for $k$ points $q_1, \ldots, q_k$, there is an $R>0$ and $B=\cup_{j=1}^k B_R(q_j)$ such that
\[
\inf\limits_{x \in \R^N \backslash B} W(x)=:m >0.
\]
Then given any number $\frac{N}{2} < \mu < N+2s$, there exists a uniform positive constant $C=C(\mu, R)$ independent of $k$ such that
for any $\varphi \in H^{2s}(\R^N)\cap L^\infty(\R^N)$ and $g$ satisfying (\ref{mp})
with
\[
\|\rho^{-1} g\|_{L^\infty(\R^N)} < +\infty, \qquad \text{where }\rho(x)=\sum\limits_{j=1}^k \frac1{(1 + |x-q_j|)^\mu},
\]
we have the validity of the estimate
\[
\|\rho^{-1} \varphi\|_{L^\infty(\R^N)} \le C \left(\|\varphi\|_{L^\infty(B)} + \|\rho^{-1} g\|_{L^\infty(\R^N)}\right).
\]

\end{lemma}

\begin{proof}
We rewrite (\ref{mp}) as
\[
(-\Delta)^s \varphi + \widetilde{W} \varphi =\tilde{g},
\]
where $\tilde{g}=(m-W)\chi_B \varphi + g$, $\widetilde{W}=m\chi_B + W(1-\chi_B)$ and $\chi_B$ is the characteristic function on $B$.
By careful calculation, it is deduced that
\[
|\tilde{g}(x)|  \le C\|\varphi\|_{L^\infty(B)} + \|\rho^{-1}g\|_{L^\infty(\R^N)} \rho\le  M\rho
\]
where
\[
\begin{split}
M=&\  C\|\varphi\|_{L^\infty(B)}\sup\limits_{x \in B}\left( \sum\limits_{j=1}^k \frac1{(1 + |x-q_j|)^\mu}  \right)^{-1} + \|\rho^{-1}g\|_{L^\infty(\R^N)}\\
\le &\ C \|\varphi\|_{L^\infty(B)}\max_{1\le j\le k} \sup\limits_{x \in B_R(q_j)}\left(  \frac1{(1 + |x-q_j|)^\mu}  \right)^{-1}+ \|\rho^{-1}g\|_{L^\infty(\R^N)}\\
\le &\ C \|\varphi\|_{L^\infty(B)}(1 + R^\mu)+ \|\rho^{-1}g\|_{L^\infty(\R^N)}\\
\le &\ C(\mu, R) \left(\|\varphi\|_{L^\infty(B)}+ \|\rho^{-1}g\|_{L^\infty(\R^N)} \right).
\end{split}
\]

From Lemma \ref{2.2} with $0 < \mu < N+2s$, the positive solution $\varphi_0$ to the problem
\[
(-\Delta)^s \varphi_0 + m \varphi_0 = \frac1{(1 + |x|)^\mu}
\]
satisfies $\varphi_0 = O(|x|^{-\mu})$ as $|x| \rightarrow +\infty$.
Since $\inf\limits_{x \in \R^N} \widetilde{W}(x) \geq m$ obviously, we have
\[
\left((-\Delta)^s + \widetilde{W}   \right) \bar{\varphi} \ge M \sum\limits_{j=1}^k \frac1{(1 + |x-q_j|)^\mu}
\]
where $\bar{\varphi}(x) =M \sum\limits_{j=1}^k \varphi_0(x -q_j)$.
Setting $\psi=\varphi - \bar{\varphi}$, one find that
\[
(-\Delta)^s \psi + \widetilde{W} \psi =\bar{g} \le 0.
\]
 Using Lemma \ref{2.4} we get that $\varphi \le \bar{\varphi}$. Arguing similarly for $-\varphi$,  we get that $|\varphi| \le \bar{\varphi}$.
Then it holds that
\[
\begin{split}
\|\rho^{-1}\varphi\|_{L^\infty(\R^N)} & \le \|\rho^{-1}\bar{\varphi}\|_{L^\infty(\R^N)}=M\left\|\left( \sum\limits_{j=1}^k \frac1{(1 + |x-q_j|)^\mu} \right )^{-1} \sum\limits_{i=1}^k \varphi_0(x
-q_i)\right\|_{L^\infty(\R^N)}\\
& \leq C M \left\|\left(\sum\limits_{j=1}^k \frac1{(1 + |x-q_j|)^\mu} \right )^{-1} \sum\limits_{i=1}^k \frac1{(1 + |x-q_i|)^\mu}\right\|_{L^\infty(\R^N)}\\
& \leq C M.
\end{split}
\]
 The desired estimate follows right now.
\end{proof}

Examining the above proof, we can deduce the following immediately.
\begin{corollary}
Let $\rho(x)$ be defined as in the previous lemma. Assume that $\varphi \in H^{2s}(\R^N)$ satisfies the problem (\ref{mp}) and that
\[
\inf\limits_{x \in \R^N} W(x)=:m >0.
\]
Then we have that $\varphi \in L^\infty(\R^N)$ and it satisfies
\begin{equation}
\|\rho^{-1} \varphi\|_{L^\infty(\R^N)} \le C  \|\rho^{-1} g\|_{L^\infty(\R^N)}
\end{equation}
where $C=C(\mu)$ independent of $k$.
\end{corollary}

\begin{remark}
We build these results for any $\frac{N}2 < \mu < N+2s$, but for our purpose, from now on we choose
$$\bm{\mu = \frac{N}{2}-\F{m}{N+2s}+1+\sigma} \in (\frac{N}2,  N+2s).$$  Here $\sigma>0$ is small enough.
\end{remark}

Due to the symmetry, we define $\Omega_j$ as follows
\[
\Omega_j=\left\{ y =(y',y'') \in \R^2 \times \R^{N-2}: \quad \left\langle \frac{y'}{|y'|}, \frac{q_j}{|q_j|}\right\rangle \ge \cos\frac{\pi}{k}\right\},
\]
and introduce the following estimate for later use. For any $\beta\geq \F{N+2s-m}{N+2s}$ and fixed $\ell$, as $k\to\infty$, it holds that
\begin{equation*}
\sum_{i\ne \ell}\frac1{|q_i-q_\ell|^\beta}=\F{1}{2^\beta}\sum_{i\neq \ell} \F{1}{r^\beta \sin^\beta\F{|i-\ell|\pi}{k}}
\le \frac{C k^\beta}{r^\beta}\sum_{i=1}^k \frac1{i^\beta}
\leq
\begin{cases}
\F{Ck^\beta}{r^\beta}=O(r^{-\F{m\beta}{N+2s}}) \qquad & \beta>1, \smallskip \\
\F{Ck^\beta\ln k}{r^\beta}=O(r^{-\F{m\beta}{N+2s}}\ln r) \quad & \beta=1, \smallskip \\
\F{C k}{r^\beta}=O(r^{-(\beta-\F{N+2s-m}{N+2s})}) \quad & \beta<1.
\end{cases}
\end{equation*}

\begin{remark}\label{r1}
It holds that
\[
\rho(x) \le C + C\sum\limits_{j=2}^k \frac1{|q_1-q_j|^{\frac{N}{2}-\F{m}{N+2s}+1+\sigma}} \leq C + C\left(\frac{k}{r}\right)^{\frac{N}{2}-\F{m}{N+2s}+1+\sigma} \le C.
\]
according to Lemma \ref{laa0}. Also we easily have
\begin{equation*}
\int_{\Omega_1} \rho^2 \leq \int_{\Omega_1} \left( \frac1{(1 + |x-q_1|)^{\frac{N}{2}-\F{m}{N+2s}+1+\sigma}} + \frac1{(1 + |x-q_1|)^{\frac{N}{2}+\sigma}}\sum_{j=2}^k
\frac1{|q_1-q_j|^{1-\F{m}{N+2s}}} \right)^2 \mathrm dx \leq C.
\end{equation*}
\end{remark}
In what follows, we use $\|f\|_*$ to mean $\|\rho^{-1} f\|_{L^\infty(\R^N)}$ for convenience, i.e.
\begin{equation*}
\|f\|_* = \|\rho^{-1} f\|_{L^\infty(\R^N)} = \sup_{x\in\R^N} \left(\sum_{j=1}^k \frac1{(1 + |x-q_j|)^{\frac{N}{2}-\F{m}{N+2s}+1+\sigma}}\right)^{-1}f(x) .
\end{equation*}
A useful fact is that if $f$, $g \in L^2(\R^N)$ and $F=T_m(f), G=T_m(g)$, then the following holds
\[
\int_{\R^N} G (-\Delta)^s F - \int_{\R^N} F (-\Delta)^s G = -\int_{\R^N} T_m(f) g + \int_{\R^N} f T_m(g) =0
\]
since the kernel $k$ is radially symmetric.

\section{Ansatz}\label{s3}
\setcounter{equation}{0}

In this section, we set up the approximation solution and estimate the corresponding error term.

By a solution of the problem
\[
(-\Delta)^s u + V u - u^p=0 \qquad \mbox{in} \  \R^N,
\]
we mean a $u \in H^{2s}(\R^N) \cap L^\infty(\R^N)$ such that the above equation is satisfied. Let us observe that it suffices to solve
\begin{equation}\label{rp}
(-\Delta)^s u + V u - u_+^p=0 \qquad \mbox{in} \ \ \R^N
\end{equation}
where $u_+ =\max\{u , 0\}$ with the help of Lemma \ref{2.4}.

We look for a solution $u$ of the form
\[
u=W +\varphi, \quad W =\sum\limits_{j=1}^k W_j, \quad W_j =w(x-q_j)
\]
where $\varphi \in H_s$ is a small function, disappearing as $k \rightarrow +\infty$.
In terms of $\varphi$, the equation (\ref{rp}) becomes
\begin{equation}\label{varphi}
(-\Delta)^s \varphi(x) + V(|x|)\varphi(x)-pW^{p-1}\varphi(x) = E + N(\varphi) \qquad \mbox{in} \ \ \R^N,
\end{equation}
where
\begin{gather*}
N(\varphi) = (W + \varphi)^p_+ - W^p- pW^{p-1}\varphi,\\
E = \sum\limits_{j=1}^k \left( 1 - V(|x|) \right)W_j + \left(\sum\limits_{j=1}^k W_j\right)^p - \sum\limits_{j=1}^k  W_j^p.
\end{gather*}

Rather than solving the problem (\ref{varphi}) directly, we shall first solve a projected version of it, precisely,
\begin{equation}\label{3.2}
\begin{cases}
(-\Delta)^s \varphi(x) + V(|x|)\varphi(x)-pW^{p-1}\varphi(x) = E + N(\varphi) + c\sum\limits_{j=1}^k Z_j \qquad \mbox{in} \ \ \R^N,\\
\ \varphi \in H_s, \smallskip\\
\displaystyle \int_{\R^N} Z_j \varphi =0, \qquad j=1, \ldots, k,
\end{cases}
\end{equation}
for some pair $(\varphi, c)$ where $ \varphi \in H^{2s}(\R^N)\cap L^\infty(\R^N), c$ is a constant,
\[
Z_j = \frac{\partial W_j}{\partial r} \qquad \text{for } j=1, \ldots, k, \qquad \text{and} \qquad |Z_j| \le \frac{C}{(1 + |x-q_j|)^{N+2s}} \text{ obviously}.
\]
After the problem (\ref{3.2}) solved, a variational process will carry out to find a suitable $r$ and then make the constant $c$ in (\ref{3.2}) be zero,
i.e. we solve the problem (\ref{varphi}).

At the end of this section, we give the estimate of $E$.
\begin{lemma}\label{ee}
It holds that
\begin{equation}
\|E\|_* \le C\left( \frac{k}{r} \right)^{\min\left\{N+2s,
(N+2s)p-\mu\right\} }+  \frac{C}{r^{N+2s-\mu}} + \frac{C}{r^m} =
o\left(\F{1}{r^{m/2}}\right).
\end{equation}
\end{lemma}

\begin{proof}
By symmetry, we just assume that $x \in \Omega_1$ in the following proof.
Obviously we know
\[
|x-q_j|\ge |x-q_1| \qquad\text{for } j=2, \ldots, k.
\]
If $|x| \ge |q_1|/2 = r/2$, then
\[
V(|x|) - 1= O\left( \frac1{|x|^m} \right)=O\left( \frac1{r^m} \right)
\]
and in this region
\[
\begin{split}
&\ \left|\rho^{-1}\sum\limits_{j=1}^k \left( 1 - V(|x|) \right)W_j\right| \le \frac{C}{r^m}\left(\sum\limits_{i=1}^k \frac{1}{(1 + |x-q_i|)^{\mu}}\right)^{-1}\sum\limits_{j=1}^k W_j\\
\le &\ \frac{C}{r^m}\left(\sum\limits_{i=1}^k \frac{1}{(1 + |x-q_i|)^{\mu}}\right)^{-1}\sum\limits_{j=1}^k \frac{1}{(1 + |x-q_j|)^{N+2s}}\\
\le &\ \frac{C}{r^m} \left(\sum_{i=1}^k \frac{1}{(1 + |x-q_i|)^{\mu}}\right)^{-1}\sum_{j=1}^k \frac{1}{(1 + |x-q_j|)^{\mu}} \\
\le &\ \frac{C}{r^m}.
\end{split}
\]
While for $|x| \leq r/2$, then
\[
|x -q_1| \ge |q_1| - |x| \ge  \frac{r}2 \quad\text{and}\quad |x-q_j|\ge \frac{r}2 \qquad\text{for}\ j=2, \ldots, k.
\]
Hence
\[
\begin{split}
&\ \left|\rho^{-1}\sum\limits_{j=1}^k \left( 1 - V(|x|) \right)W_j\right| \le C\rho^{-1}\sum\limits_{j=1}^k W_j\le  C\rho^{-1}\sum\limits_{j=1}^k \frac{1}{(1 + |x-q_j|)^{N+2s}}\\
\le &\ C\rho^{-1}\sum\limits_{j=1}^k \frac{1}{(1 + |x-q_j|)^{\mu}} \cdot \frac{1}{(1 + |x-q_j|)^{N+2s-\mu}}\\
\le &\ C\left(\sum\limits_{i=1}^k \frac{1}{(1 + |x-q_i|)^{\mu}}\right)^{-1}\sum\limits_{j=1}^k \frac{1}{(1 + |x-q_j|)^{\mu}} \cdot \frac{1}{r^{N+2s-\mu}}\\
\le &\ \frac{C}{r^{N+2s-\mu}}.
\end{split}
\]

For the other part in $E$,  we  observe that
\[
\left|  \left(\sum\limits_{j=1}^k W_j\right)^p - \sum\limits_{j=1}^k  W_j^p \right|  \le CW_1^{p-1}\sum\limits_{j=2}^k W_j +C\sum\limits_{j=2}^k W_j^p + C\left(\sum\limits_{j=2}^k W_j\right)^p.
\]
In the case of  $\mu \le (N+2s)(p-1)$,
\[
\begin{split}
&\ \rho^{-1}W_1^{p-1}\sum\limits_{j=2}^k W_j  \le  C\rho^{-1}\frac1{(1 + |x-q_1|)^{(N+2s)(p-1)}}\sum\limits_{j=2}^k  \frac1{(1 + |x-q_j|)^{N+2s}} \\
\le &\ C (1 + |x-q_1|)^{\mu}\frac1{(1 + |x-q_1|)^{(N+2s)(p-1)}}\sum\limits_{j=2}^k  \frac1{(1 + |x-q_j|)^{N+2s}} \\
\le &\ C \sum\limits_{j=2}^k  \frac1{|q_j-q_1|^{N+2s}} \le C \left(\frac{k}r\right)^{N+2s};
\end{split}
\]
Otherwise if $\mu > (N+2s)(p-1)$, then
\[
\begin{split}
&\ \rho^{-1}W_1^{p-1}\sum\limits_{j=2}^k W_j  \le  C\rho^{-1}\frac1{(1 + |x-q_1|)^{(N+2s)(p-1)}}\sum\limits_{j=2}^k  \frac1{(1 + |x-q_j|)^{N+2s}} \\
\le &\ C \rho^{-1}\frac1{(1 + |x-q_1|)^{\mu}}\sum\limits_{j=2}^k  \frac1{(1 + |x-q_j|)^{N+2s -\mu+(N+2s)(p-1)}} \\
\le &\ C \sum\limits_{j=2}^k  \frac1{|q_j-q_1|^{(N+2s)p-\mu}} \le C \left(\frac{k}r\right)^{(N+2s)p-\mu},
\end{split}
\]
where we used Lemma \ref{laa0}.
It is easy to deduce that
\begin{equation*}
\begin{split}
\rho^{-1}\sum\limits_{j=2}^k W_j^p &\le C \rho^{-1}\sum\limits_{j=2}^k \frac1{(1 + |x-q_j|)^{(N+2s)p-\mu}}\frac{1}{(1 + |x-q_1|)^{\mu}}\\
 &\le C \sum\limits_{j=2}^k  \frac1{|q_j-q_1|^{(N+2s)p-\mu}} \le C \left(\frac{k}r\right)^{(N+2s)p-\mu}
\end{split}
\end{equation*}
and
\begin{equation*}
\begin{split}
\rho^{-1}\left(\sum\limits_{j=2}^k W_j\right)^p &\le C\rho^{-1}\left( \sum\limits_{j=2}^k \frac1{(1 + |x-q_j|)^{N+2s-\frac{\mu}{p}}} \frac{1}{(1 + |x-q_1|)^{\frac{\mu}{p}}}\right)^p \\
&\le C \left(\sum\limits_{j=2}^k
\frac1{|q_j-q_1|^{N+2s-\frac{\mu}{p}}}\right)^p \le C
(\frac{k}r)^{(N+2s)p-\mu}.
\end{split}
\end{equation*}

 The condition (\ref{1.8}) of $m$ leads obviously to $N+2s-\frac{\mu}{p}>1,\ N+2s-\mu>\F{m}{2}$ and $(N+2s)p-\mu>\F{N+2s}{2}$.
Thus we get the desired result by combining these above estimates.
\end{proof}

\section{Linearized theory}\label{s4}
\setcounter{equation}{0}

This section is devoted to solve a projected linear problem.

We consider the linear problem of finding $\varphi \in H^{2s}(\R^N)$ such that for certain constant $c$, we have
\begin{equation}\label{ls}
\begin{cases}
(-\Delta)^s \varphi + V(|x|)\varphi-pW^{p-1}\varphi = g + c\sum\limits_{j=1}^k Z_j \qquad\qquad \mbox{in} \ \R^N,\\
\ \varphi \in H_s, \smallskip\\
\displaystyle\int_{\R^N} Z_j \varphi =0, \qquad j=1, \ldots, k.
\end{cases}
\end{equation}
The constant $c$ is uniquely determined in terms of $\varphi$ and $g$ when $k$ is sufficient large from the equation
\begin{equation}\label{c}
\begin{split}
c\int_{\R^N} \sum\limits_{j=1}^k Z_j Z_1 &= \int_{\R^N}\left[ (-\Delta)^s \varphi + V(|x|)\varphi-pW^{p-1}\varphi \right]Z_1 - \int_{\R^N} gZ_1\\
&=\int_{\R^N}\left[ (-\Delta)^s Z_1 + V(|x|)Z_1-pW^{p-1}Z_1\right]\varphi + O(\|g\|_*)\int_{\R^N} \rho |Z_1|\\
&=\int_{\R^N} \left[(V-1) + p(W_1^{p-1}- W^{p-1})  \right]Z_1\varphi +O(\|g\|_*),
\end{split}
\end{equation}
where we use Lemma \ref{laa0} to obtain
\[
\int_{\R^N} \rho |Z_1| \le C\left(1 + \sum\limits_{j=2}^k \frac1{|q_1-q_j|^{\mu}}\right)\int_{\R^N} \frac{1}{(1+|x|)^{N+2s}} dx \le C.
\]
By direct calculation, it is easy to see that
\begin{equation}\label{z1}
\int_{\R^N} Z_1^2 = \int_{\R^N} \left(\F{\partial w(x-q_1)}{\partial r}\right)^2 dx=\int_{\R^N} \left(\F{\partial w(x-q_1)}{\partial x_1}\right)^2 dx =\frac1N\int_{\R^N} (w'(|x|))^2dx,
\end{equation}
and
\begin{align}
&\ \sum\limits_{j=2}^k \int_{\R^N} Z_jZ_1 = \sum\limits_{j=2}^k \int_{\R^N} w'(|x-q_1|) \frac{x-q_1}{|x-q_1|} \cdot (-\frac{q_1}r)w'(|x-q_j|) \frac{x-q_j}{|x-q_j|} \cdot (-\frac{q_j}r)\mathrm dx \nonumber\\
=&\ \sum\limits_{j=2}^k \int_{\R^N} w'(|x|) \frac{x^1}{|x|} w'(|x+q_1-q_j|) \frac{x+q_1-q_j}{|x+q_1-q_j|} \cdot \frac{q_j}r \mathrm dx \nonumber\\
=&\ \sum\limits_{j=2}^k \left(\int_{\left\{ |x| \leq \frac12|q_1 -q_j|\right\}}  + \int_{\left\{ |x| \geq \frac12|q_1 -q_j|\right\}} \right)w'(|x|)\frac{x^1}{|x|} w'(|x+q_1-q_j|) \frac{x+q_1-q_j}{|x+q_1-q_j|} \cdot \frac{q_j}r \mathrm dx \nonumber\\
\leq &\ C\sum\limits_{j=2}^k \frac1{|q_1-q_j|^{N+2s}} \int_{\R^N} |w'(|x|)|dx \leq C\sum\limits_{j=2}^k \frac1{|q_1-q_j|^{N+2s}} = O\left((\frac{k}r)^{N+2s}\right), \label{zj}
\end{align}
where $x=(x^1, \ldots, x^N)$.  It implies that $\{Z_j\}_{j=1}^k$ is approximately orthogonal provided $k$ large enough because of the symmetry.\par

As to the first  term in the right hand side of (\ref{c}), we do the following analysis.
\begin{align}
&\ \left| \int_{\Omega_1} (V(|x|)-1)Z_1 \varphi \right| \le C\|\varphi\|_* \int_{\Omega_1} |V(|x|)-1| \rho \frac1{(1 + |x-q_1|)^{N+2s}} \mathrm dx \nonumber \\
\le &\ C\|\varphi\|_* \left(\int_{\{x\in \Omega_1|  |x| \ge |q_1|/2\}} + \int_{\{  x\in \Omega_1| |x| \le |q_1|/2\}} \right)|V(|x|)-1| \rho \frac1{(1 + |x-q_1|)^{N+2s}}\mathrm dx  \nonumber\\
\le &\ C\|\varphi\|_* \int_{\{ x\in \Omega_1| |x| \ge |q_1|/2\}} \frac1{|x|^m}\frac1{(1 + |x-q_1|)^{N+2s}}  \sum\limits_{j=1}^k \frac1{(1 + |x-q_j|)^{\mu}} \mathrm dx \nonumber\\
& + C\|\varphi\|_* \int_{\{x\in \Omega_1| |x| \le |q_1|/2\}}\frac1{(1 + |x-q_1|)^{N+2s}}  \sum\limits_{j=1}^k \frac1{(1 + |x-q_j|)^{\mu}}\mathrm dx  \nonumber \\
 \le &\ C\|\varphi\|_* \frac1{r^m} \int_{\{x\in \Omega_1 |  |x| \geq |q_1|/2\}} \frac1{(1 + |x-q_1|)^{N+2s}}  \left(\frac1{(1 + |x-q_1|)^{\mu}} +\sum\limits_{j=2}^k \frac1{|q_1-q_j|^{\mu}} \right) \mathrm dx \nonumber \\
 &+  C\|\varphi\|_*\int_{\{x \in \Omega_1|  |x| \leq |q_1|/2\}} \frac1{r^{N+2s}}\left(\frac1{(1 + |x-q_1|)^{\mu}} +\sum\limits_{j=2}^k \frac1{|q_1-q_j|^{\mu}} \right)\mathrm dx \nonumber \\
 \le &\ C\|\varphi\|_*\left( \frac1{r^m}+ \frac1{r^{m}} (\frac{k}r)^{\mu} + \frac1{r^{\mu+2s }} + \frac1{r^{2s}} (\frac{k}r)^{\mu}\right)=o(\|\varphi\|_*), \qquad \mbox{as} \ k \rightarrow +\infty. \label{1.1}
\end{align}
In addition, note that, for any $j\neq 1$, $\ell\neq 1$ and $j\neq \ell$,
\begin{align*}
&\ \int_{\Omega_\ell} \F{\mathrm dx}{(1+|x-q_1|)^{N+2s}(1+|x-q_j|)^\mu} \\
\leq&\ \F{C}{|q_j-q_1|^{1-\F{m}{N+2s}}}\int_{\Omega_\ell} \left[ \F{1}{(1+|x-q_1|)^{\F{3}{2}N+2s+\sigma}} + \F{1}{(1+|x-q_j|)^{\F{3}{2}N+2s+\sigma}}\right] \mathrm dx \\
\leq &\ \F{C}{|q_j-q_1|^{1-\F{m}{N+2s}}}\left[ \F{1}{|q_\ell-q_1|^{\F{N}{2}+2s+\sigma}} + \F{1}{|q_\ell-q_j|^{\F{N}{2}+2s+\sigma}} \right],
\end{align*}
where Lemma \ref{l6.3} is used in the first inequality. It is checked that
\begin{align}
&\ \left| \int_{\R^N\setminus \Omega_1} (V-1)Z_1 \varphi \right| \le C\|\varphi\|_* \sum_{\ell=2}^k \int_{\Omega_\ell} |V(|x|)-1|  \frac{\rho(x)}{(1 + |x-q_1|)^{N+2s}} \mathrm dx  \nonumber \\
\leq &\ C\|\varphi\|_* \sum_{\ell=2}^k \int_{\Omega_\ell} \frac{1}{(1 + |x-q_1|)^{N+2s}}\left(\frac1{(1 + |x-q_1|)^{\mu}}+\frac1{(1 + |x-q_\ell|)^{\mu}} + \sum_{\substack{j=2 \\ j\neq\ell}}^k \frac1{(1 + |x-q_j|)^{\mu}}\right) \mathrm dx\nonumber\\
\leq&\ C\|\varphi\|_*\left( \sum_{\ell=2}^k \F{1}{|q_\ell-q_1|^{\mu+2s}} + \sum_{\ell=2}^k\F{1}{|q_\ell-q_1|^{\mu+2s-\sigma}} +  \sum_{\substack{\ell,j=2 \\ j\neq\ell}}^k \F{C}{|q_j-q_1|^{1-\F{m}{N+2s}}} \F{1}{|q_\ell-q_1|^{\F{N}{2}+2s+\sigma}}\right) \nonumber\\
 & + C\|\varphi\|_*\sum_{\substack{\ell,j=2 \\ j\neq\ell}}^k \F{C}{|q_j-q_1|^{1-\F{m}{N+2s}}} \F{1}{|q_\ell-q_j|^{\F{N}{2}+2s+\sigma}} \nonumber \\
\le &\ C\|\varphi\|_* \left(\F{k}{r}\right)^{\F{N}{2}+2s}. \label{1.2}
\end{align}
Thus from (\ref{1.1}) and (\ref{1.2}) we get that
\begin{equation*}
\int_{\mathbb R^N} (V-1)Z_1 \varphi = o(\|\varphi\|_*) .
\end{equation*}

When $1<p\leq 2$, it holds that
\begin{align*}
&\ \left| \int_{\Omega_1} (W_1^{p-1}- W^{p-1}) Z_1\varphi \mathrm dx \right|  \leq C \|\varphi\|_* \int_{\Omega_1}  \left( \sum\limits_{j=2}^k W_j\right)^{p-1} \rho |Z_1| \mathrm dx \\
\leq &\ C \|\varphi\|_* \left( \sum_{j=2}^k \F{1}{|q_1-q_j|^{N+2s}}\right)^{p-1} \int_{\Omega_1} \left[ \F{1}{(1+|x-q_1|)^\mu} + \sum_{j=2}^k \F{1}{|q_1-q_j|^\mu} \right] \F{1}{(1+|x-q_1|)^{N+2s}} \mathrm dx \\
\leq &\ C \|\varphi\|_* \left(\F{k}{r}\right)^{(N+2s)(p-1)},
\end{align*}
and, similar to (\ref{1.2}),
\begin{align*}
&\ \left| \int_{\R^N\setminus\Omega_1} (W_1^{p-1}- W^{p-1}) Z_1\varphi  \right| dx \leq C \|\varphi\|_* \sum_{\ell=2}^k \int_{\Omega_\ell}  \left( \sum\limits_{j=2}^k W_j\right)^{p-1} \rho |Z_1|dx \\
\leq &\  C \|\varphi\|_* \sum_{\ell=2}^k \int_{\Omega_\ell} \left[\F{1}{(1+|x-q_\ell|)^{(N+2s)(p-1)}}+\left(\F{k}{r}\right)^{(N+2s)(p-1)}\right] \\
&\ \qquad\qquad\qquad \cdot \left[\F{1}{(1+|x-q_1|)^\mu}+\F{1}{(1+|x-q_\ell|)^\mu}+\sum_{\substack{j=2 \\j \neq \ell}}^k\F{1}{(1+|x-q_j|)^\mu}\right]\F{\mathrm dx}{(1+|x-q_1|)^{N+2s}} \\
\leq &\ C \|\varphi\|_* \left(\F{k}{r}\right)^{\F{N}{2}+2s}.
\end{align*}
Thus we obtain, from the above two estimates, that
\begin{equation*}
\left| \int_{\mathbb R^N} (W_1^{p-1}- W^{p-1}) Z_1\varphi \mathrm dx \right| \leq C \|\varphi\|_* \left(\F{k}{r}\right)^{\min\left\{(N+2s)(p-1),\F{N}{2}+2s\right\}}.
\end{equation*}
%
For the case $p >2$, with Lemma \ref{laa2},
\begin{align}
& \left| \int_{\Omega_1} (W_1^{p-1}- W^{p-1}) Z_1\varphi \mathrm dx \right| \leq C \|\varphi\|_* \int_{\Omega_1}  \left( W_1^{p-2} \sum\limits_{j=2}^k W_j + ( \sum\limits_{j=2}^k W_j)^{p-1} \right)\rho |Z_1| dx  \nonumber \\
\leq &\ C \|\varphi\|_* \left[ \sum_{j=2}^k \F{1}{|q_j-q_1|^{N+2s}} + \left(\sum_{j=2}^k \F{1}{|q_j-q_1|^{N+2s}}\right)^{p-1}\right] \leq C \|\varphi\|_* \left(\F{k}{r}\right)^{N+2s}, \label{1.3}
\end{align}
and, also similar to (\ref{1.2}),
\begin{equation}
\begin{split}
 \left| \int_{\R^N\setminus\Omega_1} (W_1^{p-1}- W^{p-1}) Z_1\varphi \right| &\leq C \|\varphi\|_* \sum_{\ell=2}^k \int_{\Omega_\ell}  \left( W_1^{p-2} \sum\limits_{j=2}^k W_j + ( \sum\limits_{j=2}^k W_j)^{p-1} \right)\rho |Z_1|   \\
&\leq  C \|\varphi\|_*  \left(\F{k}{r}\right)^{\F{N}{2}+2s}, \label{1.4}
\end{split}
\end{equation}
on account that, in $\Omega_\ell$,
\begin{align*}
W_1^{p-2} \sum_{j=2}^k W_j &\leq \F{C}{|q_\ell-q_1|^{(N+2s)(p-2)}}\left[ \F{1}{(1+|x-q_\ell|)^{N+2s}} + \sum_{j=2, j\neq \ell}^k \F{1}{|q_j-q_\ell|^{N+2s}} \right] ,\\
\left(\sum_{j=2}^k W_j\right)^{p-1} &\leq \F{C}{(1+|x-q_\ell|)^{(N+2s)(p-1)}} + C\left(\sum_{j=2, j\neq \ell}^k \F{1}{|q_j-q_\ell|^{N+2s}}\right)^{p-1}.
\end{align*}
So it is concluded from (\ref{1.3}) and (\ref{1.4}) that
\begin{align*}
\left| \int_{\R^N} (W_1^{p-1}- W^{p-1}) Z_1\varphi \mathrm dx \right| \leq C \|\varphi\|_* \left( \F{k}{r} \right)^{\F{N}{2}+2s}.
\end{align*}


Combining the above inequalities leads to the following lemma right now.
\begin{lemma}\label{ce}
If $(\varphi, c)$ solves the problem (\ref{ls}), then
\[
c=o(\|\varphi\|_*) + O(\|g\|_*).
\]

\end{lemma}

In the rest of this section we shall build a solution to the problem (\ref{ls}).
\begin{proposition}\label{sls}
Given $k$ large enough, the exists a solution $\varphi=T(g)$ to (\ref{ls}) which defines a linear operator of $g$, provided that $\|g\|_* < +\infty$. Moreover,
\[
\|\varphi\|_* \le C \|g\|_* \qquad \text{and} \qquad c \le C \|g\|_*,
\]
where the positive constant $C$ is independent of $k$.

\end{proposition}

The key difference of the proof between this proposition and Proposition 4.1 in \cite{DDW} is that now we should build an a priori estimate which is independent of $k$,
see the coming Lemma \ref{l1}. Once we get such estimate, the
remaining is just the same as that  in \cite{DDW}.

\begin{lemma}\label{l1}
Under the assumptions of Proposition \ref{sls}, there exists a positive constant $C$ independent of $k$ such that for any solution $\varphi$ with $\|\varphi\|_* < +\infty$, we have the following an a
priori estimate
\[
\|\varphi\|_* \le C \|g\|_*.
\]

\end{lemma}

\begin{proof}
We argue by contradiction. Suppose that there are $g_k$, $r_k \in \left[\frac1{C_0}k^{\frac{N+2s}{N+2s-m}}, C_0k^{\frac{N+2s}{N+2s-m}}\right]$ and $\varphi_k$  solving   (\ref{ls})
for $g=g_k, r=r_k$ with $\|g_k\|_* \rightarrow 0$ and $\|\varphi_k\|_* \ge C' >0$.
We may assume that $\|\varphi_k\|_* =1$. For simplicity, we drop the subscript $k$. \par

From the conditions of potential $V$, obviously $\inf_{\R^N} V >0$.
On the other hand, in the equation of $\varphi$,
 \[
(-\Delta)^s \varphi + (V-pW^{p-1})\varphi =g + c\sum\limits_{j=1}^k Z_j,
 \]
we find that
\begin{equation*}
\begin{split}
V(x) - pW^{p-1}(x) &\ge V(x) -C\left( \frac1{(1 + |x-q_1|)^{N+2s}} + \sum\limits_{j=2}^k \frac1{|q_j-q_1|^{N+2s}} \right)^{p-1}\\
&\ge V(x) - C\left( \frac1{(1 + |x-q_1|)^{N+2s}} + (\frac{k}r)^{N+2s} \right)^{p-1}
\ge \frac12 V(x)
\end{split}
\end{equation*}
for any $x \in \Omega_1\setminus B_R(q_1)$, which leads to
\[
\inf_{\R^N \setminus \underset{j=1}{\overset{k}{\cup}} B_R(q_j)} (V(x) - pW^{p-1}(x)) \ge \frac12 \inf_{\R^N} V(x) >0.
\]
Accordingly, by Lemma \ref{apriori} and Lemma \ref{ce}, it holds that
\begin{equation*}
\|\varphi\|_* \le C\left(\|\varphi\|_{L^\infty(\underset{j=1}{\overset{k}{\cup}} B_R(q_j))}  + \|g\|_* + |c| \bigg\|\sum\limits_{j=1}^k Z_j\bigg\|_*  \right)\\
 \le C\|\varphi\|_{L^\infty(\underset{j=1}{\overset{k}{\cup}} B_R(q_j))} + o(1),
\end{equation*}
from which we may assume that, up to a subsequence,
\begin{equation}\label{3.1}
\|\varphi\|_{L^\infty(B_R(q_1))} \ge \gamma >0.
\end{equation}
Let us set $\tilde{\varphi}(x)=\varphi(x+q_1)$, then $\tilde{\varphi}$ satisfies
\begin{equation}
(-\Delta)^s \tilde{\varphi} + V(|x+q_1|)\tilde{\varphi} - pw^{p-1}(x)\tilde{\varphi}=\tilde{g}
\end{equation}
where
\begin{equation}\label{1.5}
\begin{split}
\tilde{g}(x)=&g(x+q_1) + c\left(Z_1(x+q_1) + \sum\limits_{j=2}^k Z_j(x+q_1)  \right) \\
&+ p\left[\left(w(x) + \sum\limits_{j=2}^k w(x+q_1-q_j)\right)^{p-1} - w^{p-1}(x)  \right]\tilde{\varphi}.
\end{split}
\end{equation}
For any point $x$ in an arbitrarily compact set of $\R^N$, we have, from Remark \ref{r1}, that
\[
|g(x+q_1)| \le \|g\|_*\rho(x+q_1) \le C \|g\|_* =o(1).
\]
It is easy to see that $V(x +q_1) \rightarrow 1$,
\[
c=o(\|\varphi\|_*) + O(\|g\|_*) \rightarrow 0
\]
and
\begin{equation*}
\begin{split}
\left|Z_1(x+q_1) + \sum\limits_{j=2}^k Z_j(x+q_1)\right| &\le \frac{C}{(1 + |x+q_1|)^{N+2s}} + C\sum\limits_{j=2}^k \frac1{|x + q_1-q_j|^{N+2s}}\\
 &\le C + C\sum\limits_{j=2}^k \frac1{|q_1-q_j|^{N+2s}} \le C.
\end{split}
\end{equation*}
For the last term in (\ref{1.5}), as $1 < p\leq 2$,
\begin{equation*}
\begin{split}
& \left| \left[\left(w(x) + \sum\limits_{j=2}^k w(x+q_1-q_j)\right)^{p-1} - w^{p-1}(x)  \right]\tilde{\varphi} \right|\\
\le & C \left(\sum\limits_{j=2}^k w(x+q_1-q_j)\right)^{p-1} \le  C \left(\sum\limits_{j=2}^k  \frac1{|x+q_1 -q_j|^{N+2s}})\right)^{p-1}\\
\le & C \left(\sum\limits_{j=2}^k \frac1{|q_1-q_j|^{N+2s}}\right)^{p-1} \le  C \left(\frac{k}r\right)^{(N+2s)(p-1)},
\end{split}
\end{equation*}
while for $p >2$,
\begin{align*}
& \left| \left[\left(w(x) + \sum\limits_{j=2}^k w(x+q_1-q_j)\right)^{p-1} - w^{p-1}(x)  \right]\tilde{\varphi} \right|\\
\le & C w^{p-2}\sum\limits_{j=2}^k w(x+q_1-q_j) + C \left(\sum\limits_{j=2}^k w(x+q_1-q_j)\right)^{p-1}\\
\le &\  \sum\limits_{j=2}^k  \frac{C}{|x+q_1 -q_j|^{N+2s}} + C \left(\sum\limits_{j=2}^k  \frac1{|x+q_1 -q_j|^{N+2s}}\right)^{p-1}\\
\leq & \sum\limits_{j=2}^k \frac{C}{|q_1-q_j|^{N+2s}} +  C \left(\sum\limits_{j=2}^k \frac1{|q_1-q_j|^{N+2s}}\right)^{p-1}\\
\leq &  C(\frac{k}r)^{N+2s}+ C (\frac{k}r)^{(N+2s)(p-1)} \leq C\left(\frac{k}{r}\right)^{N+2s}.
\end{align*}
Hence $\tilde{g} \rightarrow 0$ uniformly on any compact set of $\R^N$ as $k \rightarrow \infty$. Meanwhile, from
\[
(-\Delta)^s \tilde{\varphi} + \tilde{\varphi} = \left(1-V(x+q_1)\right)\tilde{\varphi} + pw^{p-1}\tilde{\varphi} + \tilde{g},
\]
and Lemma \ref{2.3}, we obtain that
\begin{equation*}
\sup_{x\neq y}\frac{\left|\tilde{\varphi}(x) - \tilde{\varphi}(y)\right|}{|x-y|^\beta}  \le C\left(\|(1-V)\tilde{\varphi}\|_{L^\infty} + \|w^{p-1}\tilde{\varphi}\|_{L^\infty} +
\|\tilde{g}\|_{L^\infty} \right)\\
 \le C(\|\varphi\|_* + \|\tilde{g}\|_{L^\infty}) \le C
\end{equation*}
where $\beta=\min\{1, 2s\}$. Hence up to a subsequence, we may assume that $\tilde{\varphi} \rightarrow \varphi_0$ uniformly on any compact set. It is easy to observe that $\varphi_0$ satisfies
\begin{equation}
\begin{cases}
(-\Delta)^s \varphi_0 + \varphi_0-pw^{p-1}\varphi_0 =0 \qquad &\mbox{in}\ \R^N, \smallskip \\
\ \varphi_0 \in H_s, \medskip \\
\displaystyle\int_{\R^N} \frac{\partial w}{\partial x^1} \varphi_0 =0,
\end{cases}
\end{equation}
where $x=(x^1, \ldots, x^N)$. Besides, we know, from Remark \ref{r1}, that
\begin{equation*}
\int_{B_R(0)} \varphi_0^2 \leq \int_{B_R(0)}\tilde\varphi_k^2 = \int_{B_R(q_1)}\varphi_k^2 \leq \|\varphi_k\|_*^2 \int_{B_R(q_1)}\rho^2 \leq C,
\end{equation*}
which means that $\varphi_0 \in L^2(\R^N)$. Then the non-degeneracy result in \cite{FLS} implies  that $\varphi_0$ must be a linear combination of the partial derivatives $\frac{\partial w}{\partial x^i}, i=1, \ldots, N$. But the
symmetry and orthogonality condition yield that $\varphi_0 \equiv 0$, which is a contradiction to (\ref{3.1}).
The lemma is then proved.
\end{proof}

\section{The variational reduction and the proof of Theorem \ref{t2}} \label{s5}
\setcounter{equation}{0}

In this section  we first solve the intermediate nonlinear problem (\ref{3.2}), i.e.
\begin{equation*}
\begin{cases}
(-\Delta)^s \varphi(x) + V(|x|)\varphi(x)-pW^{p-1}\varphi(x) = E + N(\varphi) + c\sum\limits_{j=1}^k Z_j \qquad \mbox{in} \ \ \R^N,\\
\ \varphi \in H_s, \smallskip\\
\displaystyle \int_{\R^N} Z_j \varphi =0 \qquad \text{for any } j=1, \ldots, k.
\end{cases}
\end{equation*}
Then we solve the final nonlinear problem (\ref{varphi}) variationally.

\begin{proposition}\label{nls}
Assume that $k$ large enough, for any $r \in \left[\frac1{C_0}k^{\frac{N+2s}{N+2s-m}}, C_0k^{\frac{N+2s}{N+2s-m}} \right]$, the problem (\ref{3.2}) has a unique small solution $\varphi=\Phi(r)$ with
\[
\|\varphi\|_* \le C\left( \frac{k}{r} \right)^{\min\{N+2s, (N+2s)p-\mu\} }+  \frac{C}{r^{N+2s-\mu}} + \frac{C}{r^m}= o\left(\F{1}{r^{m/2}}\right).
\]
Furthermore, the map $r \rightarrow \Phi(r)$ is of class $C^1$, and
\[
\|\Phi'(r)\|_* \le C\left( \frac{k}{r} \right)^{\min\{N+2s, (N+2s)p-\mu\} }+  \frac{C}{r^{N+2s-\mu}} + \frac{C}{r^m}.
\]
\end{proposition}

\begin{proof}
Problem (\ref{3.2}) can be written as the fixed point problem
\[
\varphi=T(E + N(\varphi))=:\mathcal A(\varphi) \qquad  \text{for } \varphi \in H_s.
\]
Let
\[
\mathfrak{F}=\{ \varphi \in H_s \mid \|\varphi\|_* \le s_0\},
\]
where $s_0>0$ is a small number determined later.

If $\varphi \in \mathfrak{F}$,  either $1 < p \le 2$,
\[
\|N(\varphi)\|_* \le C\|\varphi^p\|_* \le C \|\varphi\|_*^p \|\rho\|^{p-1}_{L^{\infty}(\R^N)} \le C\|\varphi\|_*^p;
\]
or $p >2$,
\[
\|N(\varphi)\|_*   \le C \| \varphi^2 W^{p-2}\|_* + C \|\varphi^p\|_*
 \le C \|\varphi\|_*^2 |\rho W|_{L^{\infty}(\R^N)} + C\|\varphi\|_*^p \|\rho^{p-1}\|_{L^{\infty}(\R^N)} \le C \|\varphi\|_*^2 .
\]

By Proposition \ref{sls} and Lemma \ref{ee},
\[
\|\mathcal A(\varphi)\|_* \le C\left(\| E\|_* + \|N(\varphi)\|_* \right) \le C \|E\|_* + C(\|\varphi\|_* + \|\varphi\|_*^{p-1})\|\varphi\|_* \le s_0
\]
if we choose $C(s_0+s_0^{p-1}) \le \frac12$ and $k$ large enough such that
\[
C\left( \frac{k}{r} \right)^{\min\{N+2s, (N+2s)p-\mu\} }+  \frac{C}{r^{N+2s-\mu}} + \frac{C}{r^m} \le \frac12s_0.
\]

On the other hand, for any $\varphi_i \in H_s$, $i=1, 2$,
\[
|N(\varphi_1) - N(\varphi_2)|=|N'(t)(\varphi_1 - \varphi_2)|
\]
where $t $ lies between $\varphi_1$ and $\varphi_2$.

For $1 < p\le 2$, $|N'(t)| \le C|t|^{p-1} \le C (|\varphi_1|^{p-1} + |\varphi_2|^{p-1})$ which tells us that
\[
\begin{split}
\|N(\varphi_1)-N(\varphi_2) \|_* & \le C\|\varphi_1 - \varphi_2\|_* (\|\varphi_1\|^{p-1}_{L^{\infty}(\R^N)} + |\varphi_2|^{p-1}_{L^{\infty}(\R^N)})\\
& \le C\|\varphi_1 - \varphi_2\|_* (\|\varphi_1\|^{p-1}_* + \|\varphi_2\|^{p-1}_*)\|\rho\|^{p-1}_{L^{\infty}(\R^N)})\\
& \le C\|\varphi_1 - \varphi_2\|_* (\|\varphi_1\|^{p-1}_* + \|\varphi_2\|^{p-1}_*)\\
& \le Cs_0^{p-1}\|\varphi_1 - \varphi_2\|_*\le \frac12 \|\varphi_1 - \varphi_2\|_*
\end{split}
\]
provided $s_0$ small enough. And for $p>2$, $|N'(t)| \le C(W^{p-2}|t| + |t|^{p-1})$, from which we can deduce that
\[
\begin{split}
&\ \|N(\varphi_1)-N(\varphi_2) \|_* \\
\le&\ C\|\varphi_1 - \varphi_2\|_*\left[ \|\rho W\|_{L^{\infty}(\R^N)})(\|\varphi_1\|_* + \|\varphi_2\|_*) + (\|\varphi_1\|^{p-1}_* +
\|\varphi_2\|^{p-1}_*)\|\rho\|^{p-1}_{L^{\infty}(\R^N)})\right]\\
 \le&\ C \|\varphi_1 - \varphi_2\|_*\left( \|\varphi_1\|_* + \|\varphi_2\|_*+ \|\varphi_1\|^{p-1}_* +
\|\varphi_2\|^{p-1}_* \right)\\
\le &\ C(s_0 + s_0^{p-1})\|\varphi_1 - \varphi_2\|_* \le \frac12 \|\varphi_1 - \varphi_2\|_*
\end{split}
\]
with $s_0$ small enough.\par

Thus we obtain that $\mathcal A$ is a contraction mapping and the problem (\ref{3.2}) has a unique solution $\varphi$. Obviously according to Lemma $\ref{ee}$,
\[
\|\varphi\|_* \le C\left( \frac{k}{r} \right)^{\min\{N+2s, (N+2s)p-\mu\} }+  \frac{C}{r^{N+2s-\mu}} + \frac{C}{r^m}.
\]

For the proof of $\Phi(r) \in C^1$, please refer to \cite{DDW}. Here we don't repeat it.
\end{proof}

Next, we will use the above introduced ingredients to find existence results for the nonlinear problem (\ref{varphi}), i.e. the equation
\begin{equation}\label{pq}
(-\Delta)^s u + V(x)u - u_+^p=0.
\end{equation}
Set  the following energy  functional
\begin{equation}\label{functional}
J(u)=\frac12\int_{\R^N} u(-\Delta)^s u + V(x)u^2 - \frac1{p+1}\int_{\R^N} u_+^{p+1}
\end{equation}
whose nontrivial critical points are solutions to (\ref{p}).

We want to find a solution of (\ref{pq}) with the form $U =W+\varphi$ where $\varphi=\Phi(r)$ is found in Proposition \ref{nls}. Then it is easy to observe that
\[
(-\Delta)^s U + VU - U^p_+ = c\sum\limits_{j=1}^k Z_j.
\]
Hence we need to find suitable $r$ such that the coefficient $c=0$. The problem can be formulated variationally as follows.

\begin{lemma}\label{l2}
Let $F(r) = J(U)=J(W + \Phi(r))$, then $c=0$ if and only if $F'(r)=0$.
\end{lemma}

\begin{proof}
Assume that $\tilde{U}$ is the unique $s$-harmonic extension of $U=W + \Phi(r)$, then the well-known computation by Caffarelli and Silvestre \cite{CS} shows that
\[
F(r)=\frac12\int_{\R^{N+1}_+} |\nabla \tilde{ U}|^2 y^{1-2s}\mathrm dx \mathrm dy + \frac12\int_{\R^N} V(|x|) U^2  -\frac1{p+1}\int_{\R^N} U_+^{p+1}.
\]
So with $\partial_r U = \partial_r W + \Phi'(r)=\sum\limits_{i=1}^k Z_i + \Phi'(r)$, (\ref{z1}), (\ref{zj}) and Proposition \ref{nls},
\[
\begin{split}
F'(r)&=\int_{\R^{N+1}_+} \nabla \tilde{U} \cdot\nabla (\partial_r \tilde{ U}) y^{1-2s} + \int_{\R^N} V(|x|) U \partial_r U - \int_{\R^N} U_+^p \partial_r U\\\
&=\int_{\R^N} \left((-\Delta )^sU + V(|x|)U -U_+^p  \right)\partial_r U =c\sum\limits_{j=1}^k \int_{\R^N} Z_j \partial_r U\\
&=c\sum\limits_{i, j=1}^k \int_{\R^N} Z_jZ_i  + c\sum\limits_{ j=1}^k \int_{\R^N} Z_j \Phi'(r)\\
&=c\left(\frac{k}N\int_{\R^N} (w'(|x|))^2dx +k O((\frac{k}r)^{N+2s}) + O(\|\Phi'(r)\|_*\sum\limits_{j=1}^k \int_{\R^N} |Z_j|\rho) \right)\\
&=ck\left(\frac{1}N\int_{\R^N} (w'(|x|))^2dx + o(1)\right),
\end{split}
\]
with $k$ large enough. The proof is finished.
\end{proof}

Now our task is to find a critical of the functional $F(r)$.  We have the following expansion of $F(r)$.
\begin{proposition}
There exists $k_0$ such that for any $k \ge k_0, r \in I_0$, the following expansion holds
\begin{equation}
F(r)=k\left[A_1 + \frac{B_1}{r^m}- \frac{B_2k^{N+2s}}{r^{N+2s}} + o\left( r^{-m}  \right)\right]
\end{equation}
where $A_1$, $B_1$, $B_2$ are universal positive constants defined in Proposition \ref{EWA} and the interval $I_0$ is given by
\[
I_0=\left[\frac1{C_0}k^{\frac{N+2s}{N+2s-m}}, C_0k^{\frac{N+2s}{N+2s-m}} \right].
\]
\end{proposition}

\begin{proof}

Since $U=W+\varphi$, let us expand $J(U)$ at $W$ and get that
\[
\begin{split}
J(U)=&\ \frac12\int_{\R^N} U(-\Delta)^sU + V U^2 - \frac1{p+1}\int_{\R^N} U_+^{p+1}\\
=& J(W) + \int_{\R^N} \left[(-\Delta)^s U + VU - U_+^p  \right]\varphi - \frac12\int_{\R^N} \left(\varphi(-\Delta)^s\varphi + V \varphi^2 - p W^{p-1}\varphi^2\right)\\
& -\frac1{p+1}\int_{\R^N} \left((W+\varphi)_+^{p+1} - W^{p+1} - (p+1)W^p\varphi -\frac{p(p+1)}2 W^{p-1}\varphi^2 \right)\\
& + \int_{\R^N}\left(U_+^p-W^p-pW^{p-1}\varphi \right)\varphi.
\end{split}
\]
Since $\int_{\R^N} \varphi Z_j=0$ for all $j=1, \ldots, k$, the second term disappears.
From Remark \ref{r1}, we have
\begin{equation*}
\begin{split}
&\left|\int_{\R^N} \left(\varphi(-\Delta)^s\varphi + V \varphi^2 - p W^{p-1}\varphi^2\right)\right|\\
= &\int_{\R^N} \left|E+N(\varphi)\right| |\varphi|
\leq  C k\left(\|E\|_* + \|N(\varphi)\|_*\right) \|\varphi\|_* \int_{\Omega_1} \rho^2 \\
\leq &  C k\left(\|E\|_* + \|N(\varphi)\|_*\right) \|\varphi\|_* = k o(r^{-m}).
\end{split}
\end{equation*}
Similarly, it is easy to see that
\begin{equation*}
\begin{split}
&\left| \int_{\R^N} \left((W+\varphi)_+^{p+1} - W^{p+1}- (p+1)W^p\varphi -\F{p(p+1)}2 W^{p-1}\varphi^2 \right)  \right| \\
\le & C \int_{\R^N} |\varphi|^{\min\{p+1,3\}}
\leq  C k \|\varphi\|^{\min\{p+1,3\}}_*\int_{\Omega_1} \rho^2 =k o\left(r^{-m}  \right),
\end{split}
\end{equation*}
and
\begin{equation*}
\left|\int_{\R^N}\left(U_+^p-W^p-pW^{p-1}\varphi \right)\varphi\right|=O\left(\int_{\R^N}W^{p-1}\varphi^2 \right)=kO\left( \|\varphi\|_*^2\right)=ko(r^{-m}).
\end{equation*}
The proof is completed.
\end{proof}

\begin{proof}[Proof of Theorem \ref{t2}]
Now consider the following problem $\max_{r \in I_0}  F(r)$.
We want to verify that the maximum points lie in the interior of the interval $ I_0$. For this, let
\[
r_0=\left( \frac{(N+2s)B_2}{mB_1} \right)^{\frac1{N+2s-m}}k^{\frac{N+2s}{N+2s-m}} \in I_0
\]
for large positive constant $C_0$.
It is easy to show that for large $k$,
\[
F(r_0)= kA_1 + k^{1-\frac{(N+2s)m}{N+2s-m}}B_1\left( \frac{mB_1}{(N+2s)B_2}\right)^{\frac{m}{N+2s-m}}\frac{N+2s-m}{N+2s} + o\left(  k^{1-\frac{(N+2s)m}{N+2s-m}}  \right).
\]
On the other hand, it always holds that
\[
\begin{split}
F\left(\frac1{C_0}k^{\frac{N+2s}{N+2s-m}}\right)= &\ kA_1 + k^{1-\frac{(N+2s)m}{N+2s-m}}\left(B_1C_0^m-B_2C_0^{N+2s}  \right) +o\left(  k^{1-\frac{(N+2s)m}{N+2s-m}}  \right)\\
<&\ kA_1 + k^{1-\frac{(N+2s)m}{N+2s-m}}B_1\left( \frac{mB_1}{(N+2s)B_2}\right)^{\frac{m}{N+2s-m}}\frac{N+2s-m}{2(N+2s)}
\end{split}
\]
and
\[
\begin{split}
F(C_0k^{\frac{N+2s}{N+2s-m}})=&\ kA_1 + k^{1-\frac{(N+2s)m}{N+2s-m}}\left(\frac{B_1}{C_0^m}-\frac{B_2}{C_0^{N+2s}}  \right) +o\left(  k^{1-\frac{(N+2s)m}{N+2s-m}}  \right)\\
<&\ kA_1 + k^{1-\frac{(N+2s)m}{N+2s-m}}\frac{B_1}{C_0^m}+o\left(  k^{1-\frac{(N+2s)m}{N+2s-m}}  \right)\\
<&\ kA_1 + k^{1-\frac{(N+2s)m}{N+2s-m}}B_1\left( \frac{mB_1}{(N+2s)B_2}\right)^{\frac{m}{N+2s-m}}\frac{N+2s-m}{2(N+2s)},
\end{split}
\]
if we choose $C_0$ large enough such that
\[
B_1C_0^m-B_2C_0^{N+2s}<0, \qquad \frac{B_1}{C_0^m} < B_1\left( \frac{mB_1}{(N+2s)B_2}\right)^{\frac{m}{N+2s-m}}\frac{N+2s-m}{4(N+2s)}
\]
which can be  done because of $0 <m< N+2s. $
If we let $F(r_1)=\max_{r \in I_0}  F(r)$, then $r_1$ is an interior point of $I_0$ and thus $F'(r_1)=0$, which gives a critical point of $F(r)$.\par

Therefore Lemma \ref{l2} implies Theorem \ref{t2}.
\end{proof}

\section{Appendix: Energy expansion} \label{s6}
\setcounter{equation}{0}

In this section, the important expansion of the energy at $W$ is given. First we list the following lemmas, whose proofs can be found in \cite{WWY}.

\begin{lemma}\label{laa0}
For any $\alpha>0$,
\[
\sum_{j=1}^k \frac1{(1+|x-x_j|)^\alpha}\le  C + C \sum_{j=2}^k \frac1{|x_1-x_j|^\alpha}, \qquad \forall \ \ x\in \R^N
\]
where $C>0$ is a constant independent of $k$.
\end{lemma}

\begin{lemma}\label{laa2}
For any constant $0<\sigma<N-2$, there is a constant $C>0$, such that
\[
\int_{\R^N} \frac1{|y-z|^{N-2}}\frac1{(1+|z|)^{2+\sigma}}\mathrm dz\le \frac C{(1+|y|)^{\sigma}}.
\]
\end{lemma}

The proof of this lemma, a more general one actually,  can also be found in \cite{LN,WZ}.

\begin{lemma}\label{l6.3}
For any constant $0\leq \sigma \leq \min\{\alpha,\ \beta\}$, there is a constant $C > 0$ such that, for any $i\neq j$,
\begin{equation*}
\F{1}{(1+|x-q_i|)^\alpha}\F{1}{(1+|x-q_j|)^\beta} \leq \F{C}{|q_i-q_j|^\sigma} \left[\frac{1}{(1+|x-q_i|)^{\alpha+\beta-\sigma}}+
\frac{1}{(1+|x-q_j|)^{\alpha+\beta-\sigma}}\right].
\end{equation*}
\end{lemma}

The proof of the above lemma may be found in \cite{WZ1}.

Next we focus on the expansion of energy at $W$. Recall the positive least energy solution $w$ to (\ref{p1}).

\begin{proposition}\label{EWA}
It holds that
\begin{equation}
J(W)=k \left[A_1 + \frac{B_1}{r^m}-\frac{B_2k^{N+2s}}{r^{N+2s}} + o\left( r^{-m} + (\frac{k}{r})^{N+2s}  \right)\right],
\end{equation}
where
\[
A_1 =\left(\frac12-\frac{1}{p+1}\right)\int_{\R^N} w^{p+1}(x)dx, \quad B_1=\frac{a}2\int_{\R^N} w^2(x) dx
\]
and $B_2$ are all positive numbers.
\end{proposition}

\begin{proof}
Recall that
\[
q_j = \left(r\cos\frac{2(j-1)\pi}{k}, r\sin\frac{2(j-1)}{k}, \bm 0\right), \qquad j=1, \ldots, k,
\]
where $\bm 0$ is the zero vector in $\R^{N-2}$, $r \in \left[\frac1{C_0}k^{\frac{N+2s}{N+2s-m}}, C_0k^{\frac{N+2s}{N+2s-m}} \right]$ for a large positive constant $C_0$. By direct calculous, we get that
\[
|q_1-q_j|=2r\sin\frac{(j-1)\pi}{k}, \qquad 0 < c' \le \frac{\sin\frac{(j-1)\pi}{k}}{\frac{(j-1)\pi}{k}} \le c'',
\]
from which we can find that for any $\ell >1$,
\begin{equation}\label{sum}
\sum\limits_{j=2}^k \frac1{|q_j-q_1|^\ell}=\frac1{(2r)^\ell} \sum\limits_{j=2}^k \frac1{(\sin\frac{(j-1)\pi}{k})^\ell} = C_{\ell} \left(\F{k}{r}\right)^\ell + o\left(\left(\F{k}{r}\right)^\ell  \right),
\end{equation}
where $C_l >0.$

Denote
\[
W_j(x)=w(x-q_j), \quad j=1, \ldots, k, \qquad W(x)=\sum\limits_{j=1}^k W_j(x).
\]
\[
\begin{split}
J(W_1)&=\frac12\int_{\R^N} \left[w(x-q_1)(-\Delta)^sw(x-q_1)+V(|x|)w^2(x-q_1)\right]\mathrm dx-\frac1{p+1}\int_{\R^N}w^{p+1}(x-q_1)\mathrm dx\\
&=J_1(w)+\frac12\int_{\R^N}\left(V(|x|)-1\right)w^2(x-q_1)dx\\
&=\left(\frac12-\frac1{p+1}   \right)\int_{\R^N}w^{p+1}dx +\frac12\int_{\R^N}\left(V(|x-q_1|)-1\right)w^2(x)dx.
\end{split}
\]

For any $\alpha >0$, $x \in B_{r/2}(0)$,  since
\begin{equation}\label{exp}
\frac1{|x-q_1|^\alpha} =\frac1{|q_1|^\alpha}\left[1+O\left(\frac{|x|}{|q_1|}\right)   \right],
\end{equation}
we deduce that
\begin{align}
&\ \frac12\int_{\R^N}\left(V(|x-q_1|)-1\right)w^2(x) \mathrm dx=\left(\int_{\{x| |x|< \frac{r}{2}\}}  + \int_{\{x| |x|\ge \frac{r}{2}\}} \right)\left(V(|x-q_1|)-1\right)w^2(x) \mathrm dx \nonumber \\
=&\ \frac12\int_{\{x| |x|< \frac{r}{2}\}} \left[\frac{a}{|x-q_1|^m} + o\left(\frac1{|x-q_1|^{m}}  \right)     \right] w^2(x)dx + O\left( \int_{\{x| |x|\ge \frac{r}{2}\}} w^2(x)dx  \right) \nonumber \\
=&\ \frac{a}{2|q_1|^m}\int_{\{x| |x|< \frac{r}{2}\}} w^2(x)dx  + O\left(r^{-(m+1)}\int_{\{x| |x|< \frac{r}{2}\}} |x|w^2(x)dx  +  r^{-(N+4s)} \right) +o(r^{-m}) \nonumber \\
=&\ \frac{B_1}{r^m} + o(r^{-m})+ O\left(r^{-(N+4s)}\right)=\frac{B_1}{r^m} + o(r^{-m}),\label{vw}
\end{align}
where the positive constant $B_1 = \frac{a}2\int_{\R^N} w^2(x)\mathrm dx$.
Recall that
\[
\Omega_j=\left\{ y =(y',y'') \in \R^2 \times \R^{N-2}: \quad \left\langle \frac{y'}{|y'|}, \frac{q_j}{|q_j|}\right\rangle \ge \cos\frac{\pi}{k}\right\}.
\]
By symmetry, we can deduce that
\[
\begin{split}
J(W)=&\frac12\int_{\R^N} W(-\Delta)^s W + V(|x|)W^2 - \frac{1}{p+1}\int_{\R^N} W^{p+1}\\
=& \frac12\int_{\R^N} W\left((-\Delta)^sW + W \right) + \frac12\int_{\R^N} \left(V(|x|)-1\right)W^2- \frac{1}{p+1}\int_{\R^N} W^{p+1}\\
=&\frac12\int_{\R^N} W \sum\limits_{j=1}^k W^p_j + \frac{k}2 \int_{\Omega_1}\left(V(|x|)-1\right)W^2- \frac{k}{p+1}\int_{\Omega_1} W^{p+1}\\
=&\frac{k}2\int_{\R^N} w^{p+1} + \frac{k}{2}\sum\limits_{j=2}^k \int_{\R^N}W_1^pW_j +  \frac{k}2 \int_{\Omega_1}\left(V(|x|)-1\right)W^2- \frac{k}{p+1}\int_{\Omega_1} W^{p+1}.
\end{split}
\]
Now let us do the computations term by term.

With (\ref{exp}) and (\ref{expw}) at hand, we find that
\begin{align}
&\ \sum\limits_{j=2}^k \int_{\R^N}W_1^pW_j =\sum\limits_{j=2}^k \int_{\R^N} w^p(|x-q_1|) w(|x-q_j|)dx \nonumber\\
=&\ \sum\limits_{j=2}^k \int_{\R^N} w^p(|x|) w(|x + q_1-q_j|)dx \nonumber \\
=&\ \sum\limits_{j=2}^k \int_{\{x | |x|\le |q_1-q_j|/2 \}} w^p(|x|) \left[ \frac{A}{|x+q_1-q_j|^{N+2s}} + o\left(\frac1{|x+q_1-q_j|^{N+2s}} \right)\right]dx \nonumber\\
& + \sum\limits_{j=2}^k \int_{\{x | |x|\ge |q_1-q_j|/2 \}} O\left(\frac1{|q_1-q_j|^{(N+2s)p}} \right)w(|x + q_1-q_j|)    dx \nonumber\\
=&\ \sum\limits_{j=2}^k \frac{A}{|q_1-q_j|^{N+2s}} \int_{\R^N} w^p(x) dx +o\left(\sum\limits_{j=2}^k\frac1{|q_1-q_j|^{N+2s}}   \right)+ O\left(\sum\limits_{j=2}^k  \frac1{|q_1-q_j|^{(N+2s)p}}\right) \nonumber  \\
=&\ \sum\limits_{j=2}^k \frac{\widetilde{B}_2}{|q_1-q_j|^{N+2s}}+ o\left((\frac{k}{r} )^{N+2s}  \right)\label{inter}
\end{align}
where the positive constant $\widetilde{B}_2 = A\int_{\R^N} w^p.$

For any $x \in \Omega_1$, it is obvious  that $|x-q_j| \ge |x-q_1|$ and $|x-q_j| \geq |q_j-q_1|/2$ for $j=2, \ldots, k$. Then for any $0\le \alpha \le N+2s$,
\[
W_j(x) \le \frac{C}{\left(1 + |x-q_j|\right)^{N+2s}}  \le \frac{C}{\left(1 + |x-q_1|\right)^{\alpha}|q_1-q_j|^{N+2s-\alpha}}.
\]
Hence for any $0\le \alpha< N+2s-1$,
\begin{equation}
\sum\limits_{j=2}^k W_j(x)=O\left(\frac{1}{\left(1 + |x-q_1|\right)^{\alpha}}  \sum\limits_{j=2}^k \frac1{|q_1-q_j|^{N+2s-\alpha}}     \right)=O\left(\frac{1}{\left(1 +
|x-q_1|\right)^{\alpha}}\left(\frac{k}{r}\right)^{N+2s-\alpha}\right).
\end{equation}
Now we can deduce that
\begin{align}
&\ \frac{1}2 \int_{\Omega_1}\left(V(|x|)-1\right)W^2 =\frac{1}2 \int_{\Omega_1}\left(V(|x|)-1\right) \left(W_1 + \sum\limits_{j=2}^k W_j  \right)^2 \nonumber \\
=&\ \frac{1}2 \int_{\Omega_1}\left(V(|x|)-1\right)W_1^2 +  O\left( \int_{\Omega_1}\left|V(|x|)-1 \right|W_1\sum\limits_{j=2}^k W_j + (\frac{k}{r})^{2N+4s-2\alpha}\int_{\Omega_1} \frac1{(1 +
|x-q_1|)^{2\alpha}} \right) \nonumber \\
=&\ \frac{B_1}{r^m} + o(r^{-m})+ O\left(\left(\frac{k}{r}\right)^{N+2s}\int_{\Omega_1} \left|V(|x|)-1 \right|W_1 +(\frac{k}{r})^{N+3s}\right) \nonumber \\
=&\ \frac{B_1}{r^m} + o\left(r^{-m}+ (\frac{k}{r})^{N+2s}\right)+(\frac{k}{r})^{N+2s}O\left(\int_{\Omega_1} \left|V(|x|)-1 \right|w(x-q_1) \right) \nonumber \\
=&\ \frac{B_1}{r^m} + o\left(r^{-m}+ (\frac{k}{r})^{N+2s}\right)+ (\frac{k}{r})^{N+2s}O \left(\left(\int_{\{x | |x| \le \frac{r}2\}} + \int_{\{x | |x| \ge \frac{r}{2}\}}\right) \left|V(x)-1 \right|w(x-q_1)
\right) \nonumber \\
=&\ \frac{B_1}{r^m} + o\left(r^{-m}+ (\frac{k}{r})^{N+2s}\right)+  (\frac{k}{r})^{N+2s}O\left(\frac1{r^{s}} + \frac1{r^m}  \right) \nonumber\\
=&\ \frac{B_1}{r^m} + o\left(r^{-m}+ (\frac{k}{r})^{N+2s}\right), \label{poten}
\end{align}
where we choose $\alpha =\frac{N+s}{2}$.

For the last term in the energy $J(W)$, it is not difficult to check that
\begin{equation}\label{EW}
\begin{split}
&\frac{1}{p+1}\int_{\Omega_1} W^{p+1}=\frac{1}{p+1}\int_{\Omega_1} \left(W_1 + \sum\limits_{j=2}^k W_j    \right)^{p+1}\\
=&\frac{1}{p+1}\int_{\Omega_1}W_1^{p+1} + \int_{\Omega_1}W_1^p\sum\limits_{j=2}^k W_j +
O\left(\int_{\Omega_1} W_1^{p-1}(\sum\limits_{j=2}^k W_j)^2 \right) +O\left(\int_{\Omega_1}(\sum\limits_{j=2}^k W_j)^{p+1}    \right)\\
=&\frac{1}{p+1}\int_{\Omega_1}W_1^{p+1} + \int_{\Omega_1}W_1^p\sum\limits_{j=2}^k W_j +
   O\left(\left(\sum\limits_{j=2}^k \frac1{|q_j-q_1|^{\F{N}{2}+2s}}\right)^2\right) \\
   &+ O\left(\left(\sum\limits_{j=2}^k \frac1{|q_j-q_1|^{N+2s-\F{N+(p-1)s}{p+1}}}\right)^{p+1}\right)\\
=&\frac{1}{p+1}\int_{\R^N} w^{p+1} + \sum_{j=2}^k \frac{\widetilde{B}_2}{|q_1-q_j|^{N+2s}} + O\left( (\frac{k}r)^{N+4s}  \right) .
\end{split}
\end{equation}

Combining (\ref{inter}), (\ref{poten}) and (\ref{EW}), we get the desired expansion of energy
\begin{equation}\label{expanj}
J(W)=k\left[A_1 + \frac{B_1}{r^m}- \frac12\sum_{j=2}^k \frac{\widetilde{B}_2}{|q_1-q_j|^{N+2s}} + o\left( r^{-m} + (\frac{k}{r})^{N+2s}  \right)\right],
\end{equation}
where
\[
A_1 =\left(\frac12-\frac{1}{p+1}\right)\int_{\R^N} w^{p+1}(x)\mathrm dx, \quad B_1=\frac{a}2\int_{\R^N} w^2(x) \mathrm dx, \quad \widetilde{B}_2 = A\int_{\R^N} w^p.
\]

With (\ref{sum}) at hand, we finished the proof.

\end{proof}

\noindent {\bf Acknowledgement.} Wang is supported by NSFC
(Project11371254). Zhao is supported by NSFC (Project 11101155) and
by the Fundamental Research Funds for the Central Universities.

\end{document}